\documentclass[11pt]{article}
\usepackage{latexsym,amsfonts,amsmath,amsthm,amssymb,epsfig}
\usepackage{xcolor}
\usepackage{hyperref}
\usepackage[a4paper,total={170mm,240mm},vcentering,dvips]{geometry}

\usepackage[normalem]{ulem}

\newtheorem{thm}{Theorem}
\newtheorem{lem}{Lemma}
\newtheorem{prop}{Proposition}
\newtheorem{cor}{Corollary}

\theoremstyle{remark}
\newtheorem{rem}{Remark}

\newtheorem{openprob}{Open Problem}
\theoremstyle{definition}
\newtheorem{dfn}{Definition}

\newcommand{\new}[1]{{\color{black}{#1}}}

\newcommand{\R}{\mathbb{R}}
\newcommand{\N}{\mathbb{N}}
\newcommand{\C}{\mathbb{C}}
\newcommand{\Z}{\mathbb{Z}}
\newcommand\INT{{\rm INT}}
\newcommand\supp{{\rm supp\,}}
\newcommand\diag{{\rm diag\,}}
\newcommand\Real{{\rm Re\,}}

\title{New lower bounds for the integration of periodic functions}

\author{
David Krieg\footnote{Institut f\"ur Analysis, 
Johannes Kepler Universit\"at Linz, 
Altenbergerstrasse~69, 4040~Linz, Austria, 
Email: \href{mailto:david.krieg@jku.at}{david.krieg@jku.at}; \new{ORCID: \href{https://orcid.org/0000-0001-8180-8906}{0000-0001-8180-8906}}; 
the research of this author is supported by the 
Austrian Science Fund (FWF) Project M~3212-N.},\,
Jan Vyb\'iral\footnote{Dept.\ of 
Mathematics FNSPE, Czech Technical University 
in Prague, Trojanova 13, 12000 Prague, Czech Republic, 
Email: \href{mailto:jan.vybiral@fjfi.cvut.cz}{jan.vybiral@fjfi.cvut.cz}; \new{ORCID: \href{https://orcid.org/0000-0002-5498-2412}{0000-0002-5498-2412}}; the research of this author was supported
by the grant P202/23/04720S of the Czech Science Foundation and
by the European Regional Development Fund-Project ``Center 
for Advanced Applied Science''
(No. CZ.02.1.01/0.0/0.0/16\_019/0000778).
}}

\begin{document}

\date{May 23, 2023} 

\maketitle

\begin{abstract}
We study the integration problem on Hilbert spaces of (multivariate) periodic functions.
The standard technique to prove lower bounds for the error of quadrature rules uses bump functions and the pigeon hole principle.
Recently, several new lower bounds have been obtained using a different technique
which exploits the Hilbert space structure and a variant of the Schur product theorem.
The purpose of this paper is to 
(a) survey the new proof technique,
(b) show that it is indeed superior to the bump-function technique,
and (c) sharpen and extend the results from the previous papers.
\end{abstract}

\new{{\bf Keywords:} Numerical integration, Schur's product theorem, Bump function technique, Complexity, Small smoothness, Reproducing kernel Hilbert spaces.}

\medskip

\tableofcontents
\newpage

\section{Introduction}

\subsection{Overview}

We study the integration problem on reproducing kernel Hilbert spaces of multivariate periodic functions.
That is, we consider Hilbert spaces $H$ of complex-valued and continuous functions on $[0,1]^d$
which admit a complete orthogonal system of trigonometric monomials.
This scale of functions contains the classical Sobolev Hilbert spaces
of isotropic, mixed or anisotropic smoothness (see, e.g., \cite{DTU,Tem93})
as well as the weighted Korobov spaces which are common in tractability studies
(see, e.g., \cite{EP21} and the references therein).
\bigskip

We are interested in the error $e_n(H)$,
which is the smallest possible difference (in the operator norm) 
between the integration functional and a
quadrature rule using up to $n$ function values.
This corresponds to the worst case setting of information-based complexity,
see, e.g., \cite{NW1,NW2,TWW88}.
It is standard to prove lower bounds for $e_n(H)$ using the bump-function technique:
If we construct $2n$ bump functions in $H$ with disjoint support,
then $n$ of the supports do not contain a quadrature point
and therefore, for the \emph{fooling function} given by the 
sum of the corresponding bumps divided by its norm, the quadrature rule outputs zero
and the error of the quadrature rule is bounded below by
the integral of the fooling function. 
This simple technique goes back at least to \cite{Bak59} and
leads to optimal lower bounds for many classical spaces from the above family.
Interestingly, however, it is not sufficient to prove optimal lower bounds in the extreme cases when the functions in the Hilbert space are either very smooth (analytic) or barely continuous.
In these cases, optimal lower bounds were recently obtained using a different technique
which exploits the Hilbert space structure and a variant of the Schur product theorem,
see \cite{HKNV21,HKNV22,V20}.
In this paper, we want to (a) survey the new proof technique,
(b) show that it is indeed superior to the bump-function technique,
and (c) sharpen and extend the results from the previous papers.
The structure of this paper is as follows:

\medskip

\begin{itemize}
\item The precise setting is introduced in Section~\ref{sec:prelim}.
\new{An overview of results is given in Section~\ref{sec:results}.}
\item In Section~\ref{sec:proof-tech}, 
we first show that the bump function technique fails
in certain ranges of smoothness. 
We then describe the
new technique, which we will refer to as the Schur technique.
We do this for a more general family of reproducing kernel Hilbert spaces than the above,
see Theorem~\ref{thm:sum-of-squares}.
\item In Section~\ref{sec:general-periodic}, we present some general lower bounds on the error $e_n(H)$ for the spaces of periodic functions.
As a corollary, we also 
obtain a new result on the largest possible gap 
between the error of sampling algorithms and general algorithms for $L_2$-approximation 
on reproducing kernel Hilbert spaces with finite trace, see Section~\ref{sec:L2app}.
We improve a lower bound from \cite{HKNV22}, 
further narrowing the gap to the best known upper bounds
from \cite{DKU,KU19,NSU}.
\item In Section~\ref{sec:small}, we derive new lower bounds for \new{function spaces which combine the fractional isotropic smoothness $s=d/2$ with an additional logarithmic term.
Since this ensures that function values are well-defined
and that the functions in the space are continuous,
while isotropic smoothness $s=d/2$ without any logarithmic perturbation is not enough,
we refer to this setting as spaces of small smoothness. We consider also function spaces of mixed smoothness, where the borderline is given by $s=1/2$ independently on the dimension $d$.
Spaces combining fractional and logarithmic smoothness} have been studied, e.g., in \cite{DT,GS19,KSchV,Levy}.
\item In Section~\ref{sec:large}, we obtain lower bounds for the case of large smoothness, namely,
for analytic periodic functions.
We present a result on the curse of dimensionality,
which is in sharp contrast to a recent result from \cite{GODA},
and an asymptotic lower bound
that complements recent error bounds 
for classes of analytic functions on $\R^d$ from \cite{IKLP15,KSW17}.
\end{itemize}

\subsection{Preliminaries} 
\label{sec:prelim}

Let $H$ be a reproducing kernel Hilbert space (RKHS) on a non-empty set $D$.
That is, $H$ shall be a Hilbert space of functions $f\colon D\to \C$ such that point evaluation
\[
 \delta_x \colon H \to \C, \quad f\mapsto f(x)
\]
is continuous for each $x\in D$.
This means that for every $x\in D$, there is a function $K(x,\cdot)\in H$ such that $\delta_x = \langle \cdot, K(x,\cdot) \rangle_H$. The function
\[
 K \colon D \times D \to \C
\]
is called the reproducing kernel of $H$.
We refer to \cite{A50} and \cite{BT04} for basics on RKHSs.
We are interested in the computation of a continuous functional on $H$, 
that is,
an operator of the form
\begin{equation}\label{eq:functional}
 S_h\colon H \to \C, \quad S_h(f)\,=\, \langle f, h \rangle_H
\end{equation}
with some $h\in H$.
The main interest lies within integration functionals of the form
\[
 \INT_\mu \colon H \to \C, \quad  f \mapsto \int_D f\,{\rm d}\mu
\]
with a probability measure $\mu$ on $D$.
Clearly, if the functional $\INT_\mu$ is continuous,
then it can be written in the form \eqref{eq:functional}.
In this case, the corresponding representer $h$
satisfies
\[
 h(x) \,=\, \langle h, K(x,\cdot) \rangle_H \,=\, \int_D K(y,x)\,{\rm d}\mu(y)
 \qquad\text{for all } x\in D.
\]
We want to compute $S_h(f)$ 
using a quadrature rule
\begin{equation}\label{eq:quadrature}
 Q_n \colon H \to \C, \quad Q_n(f) = \sum_{k=1}^n a_k f(x_k)
\end{equation}
with nodes $x_k\in D$ and weights $a_k \in \C$.
The worst-case error of the quadrature rule is defined by
\[
 e(Q_n,H,S_h) \,=\, \sup_{f\in H\colon \Vert f \Vert_H \le 1} | Q_n(f) - S_h(f) |
 \,=\, \Vert Q_n - S_h \Vert_{\mathcal L(H,\C)}.
\]
We are interested in the error of the best possible quadrature rule,
which we denote by
\[
 e_n(H,S_h) \,=\, \inf_{Q_n}\, e(Q_n,H,S_h).
\]

\medskip

In this paper, we will mainly consider Hilbert spaces of multivariate periodic functions,
defined as follows.
Here, $e_k = \exp(2\pi i \langle k, \cdot\rangle)$ denotes the trigonometric monomial with frequency $k\in \Z^d$.
The Fourier coefficients of an integrable function $f\colon [0,1]^d \to \C$ are defined by
\[
 \hat f(k) \,:=\, \langle f, e_k \rangle_2 \,=\, \int_{[0,1]^d} f(x)\, e^{-2\pi i \langle k, x \rangle} \,dx,
 \quad k\in\Z^d.
\]

\begin{dfn}\label{dfn:Hper}
Let $\lambda \in \ell_1(\Z^d)$ be a non-negative sequence.
Then $H_\lambda$ denotes the set of all
continuous functions $f\in C([0,1]^d)$
with $\hat f(k)=0$ for all $k\in\Z^d$ with $\lambda_k=0$ and
\begin{equation}\label{eq:f_norm}
\|f\|_{H_\lambda}^2\,:=\,\sum_{k\in\Z^d\colon \lambda_k\ne 0}\frac{|\hat f(k)|^2}{\lambda_k}\,<\,\infty.
\end{equation}
\end{dfn}

The condition \eqref{eq:f_norm} implies that the Fourier series of $f\in H_\lambda$
converges point-wise (and even uniformly) for all $x\in\R^d$,
\[
 f(x) \,=\, \sum_{k\in\Z^d} \hat f(k)\, e_k(x), \quad x\in \R^d,
\]
and that $f$ is continuous as a 1-periodic function on $\R^d$.
Moreover, 
\[
 |f(x)| \,\le\, \sum_{k\in\Z^d} |\hat f(k)| 
 \,\le\, \Vert \lambda \Vert_1^{1/2} \cdot \Vert f\Vert_{H_\lambda}
\]
and thus,
the space $H_\lambda$ is a reproducing kernel Hilbert space.
The trigonometric monomials $(e_k)_{k\in\Z^d}$ are a complete orthogonal system in $H_\lambda$ and
the reproducing kernel of $H_\lambda$ is given by
\begin{equation}\label{eq:kernel}
K_\lambda(x,y) = \sum_{k\in\Z^d} \lambda_k e_k(x-y).
\end{equation}
We note that the space $H_\lambda$ can be defined analogously
for bounded sequences $\lambda \not\in \ell_1(\Z^d)$
as the space of all $f\in L_2([0,1]^d)$ satisfying \eqref{eq:f_norm}.
In this case, however, function evaluation is not well defined
and $H_\lambda$ it is not a reproducing kernel Hilbert space.
The most common choices (see, e.g., \cite{Tem93} or \cite{KSU15}) of the sequence $\lambda$ include
\begin{enumerate}
\item $\lambda_k = (1+|k|^{2})^{-s}$ for $s\ge 0$ and $k\in\Z^d$. In that case, $H_\lambda$ is the Sobolev space of periodic functions with isotropic smoothness.
If $s$ is a positive integer, then an equivalent norm on $H_\lambda$ is given by
\[
\|f\|^2_{H_{\lambda}}\,\asymp\, \sum_{\|\alpha\|_1\le s}\|D^{\alpha}f\|_2^2.
\]

\item $\lambda_k = \prod_{j=1}^d (1 + |k_j|^{2})^{-s}$ for $k\in\Z^d$ and $s>0$. In this way, we obtain the spaces of periodic functions with dominating mixed smoothness.
If $s\in\N$ is a positive integer, then also these spaces allow for an equivalent norm given in terms of derivatives, namely
\[
\|f\|^2_{H_{\lambda}} \,\asymp\, \sum_{\|\alpha\|_\infty\le s}\|D^{\alpha}f\|_2^2.
\]
\end{enumerate}
\smallskip

On the space $H_\lambda$, 
we consider the integration problem with respect to the Lebesgue measure $\lambda^d$ on $[0,1]^d$,
that is, we take $\mu = \lambda^d$.
In this case, the representer of the functional $\INT_\mu$ is given by the constant function 
$h \equiv \lambda_0$.
For this standard integration problem on the unit cube,
we leave out the $S_h$ in the notation of the error,
writing $e_n(H_\lambda)$ instead of $e_n(H_\lambda,S_h)$,
$e(Q_n,H)$ instead of $e(Q_n,H,S_h)$
and we simply write $\INT$ instead of $\INT_\mu$.
The initial error in this case is given by
\[
 e_0(H_\lambda) \,=\, \Vert \lambda_0 \Vert_{H_\lambda} \,=\, \INT(\lambda_0)^{\new{1/2}} \,=\, \lambda_0^{\new{1/2}}.
\]

\new{In analogy to the numbers $e_n(H_\lambda)$, we also define the sampling numbers
\[
 g_n(H_\lambda) \,=\, \inf_{\substack{x_1,\dots,x_n\in [0,1]^d\\g_1,\dots,g_n \in L_2}}\,
 \sup_{\|f\|_{H_\lambda}\le 1}\bigg\|f-\sum_{i=1}^n f(x_i)g_i\,\bigg\|_{2}
\]
and the approximation numbers
\[
 a_n(H_\lambda) \,=\, \inf_{\substack{L_1,\dots,L_n\in H_\lambda'\\g_1,\dots,g_n \in L_2}}\,
 \sup_{\|f\|_{H_\lambda}\le 1}\bigg\|f-\sum_{i=1}^n L_i(f)g_i\,\bigg\|_{2},
\]
which reflect the error of the best possible (linear) algorithm for $L_2$-approximation
using at most $n$ function values or $n$ arbitrary linear measurements, respectively.}

\smallskip

\textbf{Further notation.}
For sequences $(a_n)$ and $(b_n)$, we write $a_n \lesssim b_n$ if there is a constant $c>0$
such that $a_n \le c\, b_n$ for all but finitely many $n$.
We write $a_n \gtrsim b_n$ if $b_n \lesssim a_n$
and $a_n \asymp b_n$ if both relations are satisfied.
\new{If $H$ is a Hilbert space, $H'$ denotes the space of all continuous linear functionals on $H$.}

\subsection{\new{Results}}
\label{sec:results}

\new{As mentioned earlier, the main purpose of the paper is to survey a new proof technique
for lower bounds on the error $e_n(H_\lambda)$.
Here, we gather some of the new results obtained in this paper using the
new technique.

\medskip

The first category of results concerns the asymptotic behavior of
the errors $e_n(H_\lambda)$ and $g_n(H_\lambda)$ in the case of small smoothness, i.e.,
in the case that $\lambda$ decays barely fast enough to ensure that $H_\lambda$ is a reproducing kernel Hilbert space.
For function spaces $H_\lambda$ with fractional isotropic smoothness $d/2$ 
and logarithmic smoothness $\beta>1/2$,
i.e., if $\lambda$ is given by \eqref{eq:lambda-iso},
we obtain
\begin{equation}
\label{eq:summary1}
e_n(H_\lambda)
\,\asymp\,
g_n(H_\lambda) 
\,\asymp\,
n^{-1/2} \log^{-\beta+1/2} n
\,\asymp\,
a_n(H_\lambda)\cdot \log^{1/2} n,
\end{equation}
see Corollary~\ref{cor:lowerbound}.
Of special interest is the logarithmic gap between the sampling and the approximation numbers,
which has been shown in \cite{HKNV22} for the univariate case,
and is now extended to the multivariate case.
The lower bound on the integration error follows from a result for more general sequences $\lambda$, see
Theorem~\ref{thm:asymptotic-mulit},
and cannot be proven by the standard technique of bump functions,
see Section~\ref{sec:limits}.
We also prove the same relation \eqref{eq:summary1} 
for the (multivariate) spaces of fractional mixed smoothness $1/2$ 
and logarithmic smoothness $\beta>1/2$,
where $\lambda$ is given by \eqref{eq:lambda-mix},
see Section~\ref{sec:mix}.
\medskip

Second, we obtain new results on the numbers $e_n^*(\sigma)$ and $g_n^*(\sigma)$, 
which reflect the worst possible behavior of the $n$th minimal integration error and sampling number, respectively,
on reproducing kernel Hilbert spaces for any given sequence $(\sigma_n)_{n\in\N}$
of singular values, see Section~\ref{sec:L2app} for a precise definition.
We obtain that, up to universal constants, both these numbers are equivalent to
\[
\sigma_n^* \,:=\, \min\left\{\sigma_0,\, \sqrt{\frac1n \sum_{k\ge n} \sigma_k^2} \right\},
\]
see Corollary~\ref{cor:opt}. The new thing here is again the lower bound, the upper bound is known from \cite{DKU}.
\medskip

Thirdly, we also derive a result on the curse of dimensionality.
We obtain that the curse is present for numerical integration and $L_2$-approximation 
on classes of the form
\[
 \left\{ f \in C([0,1]^d) \,\Big|\, \sum_{k\in\Z^d} \vert \hat f(k)\vert^p \cdot g(\Vert k\Vert_\infty) \le 1 \right\}
\]
for any weight function $g\colon \N_0 \to (0,\infty)$ and $p=2$, see Theorem~\ref{thm:curse}.
This is in sharp contrast to recent tractability results for analogous classes with $p=1$,
see \cite{GODA} for numerical integration and \cite{Kri23} for $L_2$ approximation.}

\section{Proof techniques for lower bounds}
\label{sec:proof-tech}

\new{In this section, we discuss techniques to prove lower bounds for the integration error $e_n(H_\lambda)$.}
We show the limits of the bump-function technique,
\new{which is probably the most common technique,}
and describe a new method 
used in this paper.
We note that there are other proof techniques which we do not discuss here
like the technique of decomposable kernels from \cite{NW01}
and a method from \cite{SW97}.
We only compare our new approach 
with the most standard technique.

\subsection{Limits of the bump-function technique}
\label{sec:limits}

The most well-known technique to prove lower bounds for the integration problem is the use of bump functions. 
It is based on the observation (cf.\ \cite[Chapter 4]{NW1})
that a lower bound for $e(Q_n,H)$ for a fixed quadrature rule $Q_n$ 
with nodes $x_1,\hdots,x_n$
is equivalent to the construction of a function $f$, which vanishes at the
sampling points, i.e., with $f(x_1)=\hdots=f(x_n)=0$, 
and which has simultaneously large integral $\INT(f)$ and small norm $\|f\|_H.$ This leads to $Q_n(f)=0$ and the lower bound
\[
e(Q_n,H)\ge \frac{|\INT(f)-Q_n(f)|}{\|f\|_H}=\frac{|\INT(f)|}{\|f\|_H}.
\]
A function like this is called a \emph{fooling function}.
One way how to construct such a fooling function to a given $Q_n$ is to start with a \emph{bump function}, i.e., with a smooth periodic function $\varphi \in C([0,1])$ with $\supp \varphi\subset [0,1/(2n)]$ and
consider its \new{translates} $\varphi(x-j/(2n))$ for $j=0,\dots,2n-1$
(and similarly for the multivariate case). 
By the pigeonhole principle, there are at least $n$ of these dilates, which vanish at all the sampling points
$x_1,\dots,x_n$. Adding them, we then obtain the fooling function~$f$.

The use of this (rather intuitive) technique can be traced back at least to \cite{Bak59}. It is probably
the most widely used method to prove lower bounds for numerical integration and other approximation problems, where one is limited to the use of function values, cf.~\cite{Nov88,NT06,TWW88}.
The main reason for its wide use is surely that it often leads to optimal results. Note however that for example in \cite{Vyb08} it was necessary to combine the bump-function technique with the Khintchine inequality
to obtain optimal lower bounds in certain limiting cases.
\bigskip

The aim of this section is to show that the technique of bump functions
can only provide sub-optimal results
for the integration problem on $H_\lambda$ in the cases that
\begin{itemize}
\item the sequence $\lambda$ decays very slowly. 
This means that $H_\lambda$ contains functions with low smoothness,
which are just about continuous.
\item the sequence $\lambda$ decays very fast.
This means that the functions in $H_\lambda$ are very smooth or even analytic.
\end{itemize}
In the second case, it is rather obvious that the bump function technique does not work.
If $\lambda$ decays fast enough (for example if $|\lambda_k|\le c_1\exp(-c_2 |k|)$ for some $c_1,c_2>0$
and all $k\in\Z$, c.f.\ \cite[\S 25]{Bary}),
all functions in $H_\lambda$ are analytic
and since there is no analytic function with compact support (except the zero function), 
any fooling function obtained by the bump-function technique will not be contained in 
the space~$H_\lambda$.
\bigskip

We now show the sub-optimality of the bump-function technique in the case of small smoothness.
To ease the presentation, we consider only the case $d=1$,
but we note that similar results can be obtained in the multivariate case. 
We consider sequences of the form 
\begin{equation}\label{eq:border-lambda}
 \lambda_k \,=\, (1+|k|)^{-1} \log^{-2\beta}(e+|k|), \quad k\in\Z, \quad \beta>1/2,
\end{equation}
which are just about summable.
It was shown in \cite[Theorem 4]{HKNV22} and it also follows from Corollary \ref{cor:lowerbound} below
that
\begin{equation}\label{eq:opt-lower}
 e_n(H_\lambda) \,\gtrsim\, n^{-1/2} \log^{-\beta+1/2} n.
\end{equation}
This lower bound is sharp, see \cite[Proposition 1]{HKNV22} and \cite{DKU}.
Here, we are going to prove the following.

\begin{thm}\label{thm:no-bumps} Let $\lambda \in \ell_1(\Z)$ be given by \eqref{eq:border-lambda}. 
There exists an absolute constant $C>0$ such that for every even $n\ge 2$ and every $\varphi\in C([0,1])$ with ${\rm supp}\,\varphi\subset [0,1/(2n)]$ there exists $\{z_1,\dots,z_n\}\subset \{j/(2n):j=0,\dots,2n-1\}$
such that the function
\begin{equation}\label{eq:phin}
\varphi^{(n)}(x)=\sum_{j=1}^n \varphi(x-z_j)
\end{equation}
satisfies
\begin{equation}\label{eq:bump}
\frac{\int_0^1 \varphi^{(n)}(x)dx}{\left\|\varphi^{(n)}\right\|_{H_\lambda}}\le C\,n^{-1/2}(\log n)^{-\beta}.
\end{equation}
\end{thm}

Theorem \ref{thm:no-bumps} shows that no matter how we choose the 
bump function $\varphi\in C([0,1])$ with support in $[0,1/(2n)]$,
it is possible to choose the set of sampling points $X=\{x_1,\dots,x_n\}\subset[0,1]$ in such a way that leaving from the sum
\[
\sum_{j=0}^{2n-1}\varphi\left(x-\frac{j}{2n}\right)
\]
the terms which do not vanish on $X$ does not provide a fooling function sufficient to show \eqref{eq:opt-lower}.
\new{In other words, the classical bump function technique cannot yield the sharp lower bound on $e_n(H_\lambda)$ in~\eqref{eq:summary1}.}

Before we come to the proof of Theorem \ref{thm:no-bumps}, 
we need a characterization of the norm of $f$ in the space $H_\lambda$ in terms of first order differences,
which are defined for every $0<h<1$ simply by
\[
\Delta_h f(x)=f(x+h)-f(x),\quad x\in[0,1].
\]
Here, we interpret $f$ as a periodic function defined on all $\R$, which makes $\Delta_h f(x)$ indeed well-defined for all $x\in[0,1]$.
Furthermore, $\|\Delta_hf\|_2$ denotes the norm of $x\mapsto \Delta_hf(x)$ in $L_2([0,1])$.
\new{The following proposition is in principle a special case of \cite[Theorem 10.9]{DT}, where a characterization by first and higher order
differences is given for Besov spaces of generalized smoothness on $\R^d$. 
We provide a short and direct proof for reader's convenience.}

\medskip

\begin{prop}\label{prop:diff} Let $\lambda \in \ell_1(\Z)$ be given by \eqref{eq:border-lambda}. 
Then $f\in C([0,1])$
belongs to $H_\lambda$ if, and only if,
\begin{equation}\label{eq:Delta}
\|f\|_2^2+\int_0^1 \frac{[1-\log h]^{2\beta}}{h^2}\cdot \|\Delta_h f\|_2^2\, dh
\end{equation}
is finite. Furthermore, $\|f\|_{H_\lambda}^2$ is equivalent to \eqref{eq:Delta} with the constants independent of $f\in C([0,1])$.
\end{prop}

\medskip

\begin{proof}
Let $f\in C([0,1])$. 
Then
\[
\widehat{\Delta_h f} (k) \,=\,  [e^{2\pi ikh}-1] \cdot \hat f(k)
\]
for all $k\in\Z$ and, by orthonormality of $(e_k)_{k\in\Z}$ in $L_2([0,1])$,
\[
\|\Delta_h f\|_2^2 \,=\, \sum_{k\in\Z} |\hat f(k)|^2\cdot |e^{2\pi ikh}-1|^2.
\]
To simplify the notation, we denote 
\[
\omega_\beta(h)=\frac{[1-\log h]^{\beta}}{h},\quad 0<h<1
\]
and obtain
\begin{align*}
\|f\|_2^2+\int_0^1 \omega_\beta(h)^2 \|\Delta_h f\|_2^2\,dh=\sum_{k\in\Z} |\hat f(k)|^2\Bigl(1+\int_0^1 \omega_\beta(h)^2 |e^{2\pi ikh}-1|^2  dh\Bigr).
\end{align*}

For $j\in\Z$ we put
\[
\gamma_j=1+\int_0^1 \omega_\beta(h)^2 |e^{2\pi ijh}-1|^2  dh
\]
and we will show that $\gamma_j  \asymp  1/\lambda_j$ for all $j\in \Z$ with universal constants of equivalence which do not depend on $j$.
In view of \eqref{eq:f_norm}, this will finish the proof.

The estimate of $\gamma_j$ for $j=0$ can be always achieved by changing the constants of equivalence. Let $j\not =0.$ Then
\begin{align*}
\gamma_j &\ge 1+\int_{1/(4\pi |j|)}^{1/(2\pi |j|)} \omega_\beta(h)^2 |e^{2\pi ijh}-1|^2 dh\\
&\ge 1+\frac{1}{4\pi |j|}\cdot\frac{[1-\log(1/(2\pi |j|))]^{2\beta}}{1/(2\pi |j|)^2}\cdot |e^{i/2}-1|^2
\,\gtrsim\, \lambda_j^{-1}.
\end{align*}
The estimate of $\gamma_j$ from above is then obtained by
\begin{align*}
\gamma_j&=1+ \int_{0}^{1/(2\pi |j|)} \omega_\beta(h)^2 |e^{2\pi ijh}-1|^2 dh + \int_{1/(2\pi |j|)}^1 \omega_\beta(h)^2 |e^{2\pi ijh}-1|^2 dh\\
&\lesssim 1+\int_{0}^{1/(2\pi |j|)} \omega_\beta(h)^2 (j\cdot h)^2 dh + \int_{1/(2\pi |j|)}^1 \omega_\beta(h)^2dh\\
&\lesssim 1+|j|\cdot (1+\log(2\pi |j|))^{2\beta}\,\asymp\, \lambda_j^{-1},
\end{align*}
which finishes the proof.
\end{proof}

\begin{proof}[Proof of Theorem~\ref{thm:no-bumps}]
Fix $\varphi\in C([0,1])$ with ${\rm supp}\,\varphi\subset [0,1/(2n)]$ and an even integer $n\ge 2$. We set 
\[
z_1=0,\quad z_2=\frac{1}{2n},\quad z_3=\frac{4}{2n},\quad z_4=\frac{5}{2n},\quad \text{etc.},
\]
i.e.,
\[
z_{2j+1}=\frac{4j}{2n},\quad z_{2j+2}=\frac{4j+1}{2n},\quad j=0,\dots,\frac{n}{2}-1.
\]
We define $\varphi^{(n)}$ again by \eqref{eq:phin}
and obtain by H\"older's inequality
\begin{equation}\label{eq:int_phin}
\int_0^1\varphi^{(n)}(x)dx=n\int_{0}^{1/(2n)}\varphi(t)dt\le \sqrt{\frac{n}{2}}\|\varphi\|_2.
\end{equation}
To estimate $\|\varphi^{(n)}\|_{H_\lambda}$ from below, observe that if $x\in\left[\frac{4j+1}{2n},\frac{4j+2}{2n}\right]$ 
and $1/(2n)\le h\le 2/(2n)$, then $x+h\in\left[\frac{4j+2}{2n},\frac{4j+4}{2n}\right]$ and $\varphi^{(n)}(x+h)=0$. Therefore,
\[
\|\Delta_h\varphi^{(n)}\|_2^2\ge \frac{n}{2}\int_0^{1/(2n)}\varphi^2(t) dt
\]
and, using Proposition \ref{prop:diff},
\begin{align*}
\|\varphi^{(n)}\|_{H_\lambda}^2&\gtrsim \int_{1/(2n)}^{2/(2n)} \frac{[1-\log h]^{2\beta}}{h^2}\|\Delta_h\varphi^{(n)}\|_2^2\,dh\\
&\gtrsim \frac{1}{n}\cdot(\log n)^{2\beta}\cdot n^2\cdot \frac{n}{2}\cdot \int_0^{1/(2n)}\varphi^2(t) dt=\frac{n^2}{2} (\log n)^{2\beta}\|\varphi\|_2^2.
\end{align*}
Together with \eqref{eq:int_phin}, this finishes the proof.
\end{proof}

\begin{rem} We point out also another reason, why Theorem \ref{thm:no-bumps} is rather counter-intuitive.
Motivated by an explicit formula for the optimal fooling function for equi-distributed sampling points we consider (for fixed $n\in\N$) the bump function
\[
\varphi_n(t)=\sum_{k=1}^\infty \lambda_{2nk}[1-e_{2nk}(t)],\quad t\in [0, 1/(2n)]
\]
and $\varphi_n(t)=0$ if $t\not \in[0, 1/(2n)]$. Then indeed $\varphi_n\in C([0,1])$ and
\begin{equation}\label{eq:Phi}
\Phi_n(t)=\sum_{j=0}^{2n-1}{\varphi_n(t-j/(2n))}=\sum_{k=1}^\infty \lambda_{2nk}[1-e_{2nk}(t)],\quad t\in[0,1].
\end{equation}
Note, that $\Phi_n(t)$ indeed vanishes in the points $t=j/(2n), j=0,\dots,2n.$

Finally, an easy calculation reveals that
\begin{equation}\label{eq:referee1}
\frac{\INT(\Phi_n)}{\|\Phi_n\|_{H_{\lambda}}} \,\asymp\, \biggl(\sum_{k=1}^\infty \lambda_{2nk}\biggr)^{1/2} \,\asymp\, n^{-1/2}(\log n)^{-\beta+1/2}
\end{equation}
with constants of equivalence independent of $n$. Theorem \ref{thm:no-bumps} therefore shows, that removing some of the bumps from the sum in \eqref{eq:Phi}
can actually increase the norm of such a function. \new{Indeed, if $\varphi^{(n)}$ is the function constructed in Theorem \ref{thm:no-bumps} using the $\varphi_n$
from \eqref{eq:Phi}, then the integrals of $\Phi_n$ and $\varphi^{(n)}$ are comparable and \eqref{eq:referee1} together with \eqref{eq:bump}
shows that $\|\varphi^{(n)}\|_{H_\lambda}/\|\Phi_n\|_{H_\lambda}$ grows (at least) as $\sqrt{\log(n)}$ if $n$ tends to infinity.}
\end{rem}

\subsection{The Schur technique}
\label{sec:Schur}

So how can we prove lower bounds for the integration problem
in the cases where the bump-function technique does not work?
The recent results for small smoothness and for analytic functions have been obtained using a certain modification of the classical
Schur product theorem on the entry-wise product of positive semi-definite matrices.
We will describe this technique now in the general setting of Section~\ref{sec:prelim}.
That is, we are given a RKHS $H$ with kernel $K$ on a domain $D$ and 
a functional $S_h$
represented by $h\in H$.
\bigskip

The first ingredient of this technique is a characterization
of lower bounds on numerical integration via the positive definiteness of certain matrices
involving the kernel $K$ and the representer $h$ of the integral
from \cite{HKNV22}.

\begin{prop}\label{prop:charact}
Let $H$ be a RKHS on a domain $D$ with the kernel $K\colon D\times D\to \C$ and let
$h\in H$. 
Then, for every $\alpha>0$,
\begin{equation}\label{eq:enlower}
e_n(H,S_h)^2\,\ge\, \|h\|_H^2-\alpha^{-1}
\end{equation}
if, and only if, the matrix 
\[
 \big(K(x_j,x_k)-\alpha \overline{h(x_j)}h(x_k)\big)_{j,k=1}^n
\] 
is positive semi-definite for all 
$\{x_1,\dots, x_n\} \subset D.$
\end{prop}
\smallskip

\begin{proof}
Let $Q_n$ be given by \eqref{eq:quadrature}. Then we denote $a=(a_1,\dots,a_n)^*$, ${\mathbf h}=(h(x_1),\dots,h(x_n))^*$ and ${\mathbf K}=(K(x_j,x_k))_{j,k=1}^n$
and obtain
\begin{align}\label{eq:proofProp3}
e(Q_n,H,S_h)^2&=\sup_{\|f\|_H\le 1}\biggl|\biggl\langle f, \sum_{k=1}^n a_k\delta_{x_k}-h\biggr\rangle\biggr|^2=\biggl\|\sum_{k=1}^n a_k K(x_k,\cdot)-h\biggr\|_H^2
=\|h\|_H^2-2\Real(a^*{\mathbf h})+a^*{\mathbf K}a.
\end{align}

Let us assume now that ${\mathbf K}-\alpha {\mathbf h}{\mathbf h}^*$ is positive-semidefinite. If $a^*{\mathbf K}a=0$, then also $a^*{\mathbf h}{\mathbf h}^*a=|a^*{\mathbf h}|^2=0$
and \eqref{eq:proofProp3} implies that $e(Q_n,H,S_h)^2\ge \|h\|_H^2.$ If $a^*{\mathbf K}a$ is positive, then we continue \eqref{eq:proofProp3} by
\begin{equation}\label{eq:proofProp3'}
e(Q_n,H,S_h)^2=\|h\|_H^2+\left|\frac{a^*{\mathbf h}}{\sqrt{a^*{\mathbf K}a}}-\sqrt{a^*{\mathbf K}a}\right|^2-\frac{|a^*{\mathbf h}|^2}{a^*{\mathbf K}a}
\ge \|h\|_H^2-\frac{|a^*{\mathbf h}|^2}{a^*{\mathbf K}a}.
\end{equation}
We use that $a^*{\mathbf K}a-\alpha a^*{\mathbf h}{\mathbf h}^*a=a^*{\mathbf K}a-\alpha |a^*{\mathbf h}|^2\ge 0$ and take the infimum over all quadrature formulas $Q_n$
and obtain that $e_n(H,S_h)^2=\inf_{Q_n}e(Q_n,H,S_h)^2\ge \|h\|_H^2-\alpha^{-1}.$

On the other hand, assume that \eqref{eq:enlower} holds. Then $e(Q_n,H,S_h)^2\ge \|h\|_H^2-\alpha^{-1}$ for every quadrature formula $Q_n$ with arbitrary nodes $\{x_1,\dots,x_n\}\subset D$
and arbitrary weights $a_1,\dots,a_n\in\C$. If $a^*{\mathbf K}a=0$, then it follows from \eqref{eq:proofProp3} that $2\Real(a^*{\mathbf h})\le \alpha^{-1}$ holds for $a$ and all its complex multiples. Hence, $a^*{\mathbf h}=0$
and $a^*{\mathbf K}a-\alpha a^*{\mathbf h}{\mathbf h}^*a=0.$ If $a^*{\mathbf K}a$ is positive, then we can assume (possibly after rescaling $a$ with a non-zero $t\in\C$)
that $a^*{\mathbf h}=a^*{\mathbf K}a$, in which case \eqref{eq:proofProp3'} becomes an identity.
Hence, $a^*{\mathbf K}a\ge \alpha |a^*{\mathbf h}|^2$
and the result follows.
\end{proof}

The second ingredient is a lower bound on the entry-wise square of
a positive semi-definite matrix related to the Schur product theorem,
which was proven in \cite{V20}.
If $M=(M_{i,j})_{i,j=1}^n\in \C^{n\times n}$, then we denote by $\overline M=(\overline{M}_{i,j})_{i,j=1}^n$ the matrix with complex conjugated entries
and by $M\circ\overline M$ the matrix with entries $|M_{i,j}|^2$. Furthermore, $\diag M$ is the column vector of the diagonal entries of $M$.

\begin{prop}\label{prop:Schur}
Let $M\in \C^n\times \C^n$ be a self-adjoint positive semi-definite matrix. Then
\[
M\circ \overline{M} - \frac{1}{n}(\diag M)(\diag M)^T
\]
is also positive semi-definite.
\end{prop}

Proposition \ref{prop:charact} and Proposition \ref{prop:Schur} can be easily combined to obtain lower bounds for numerical integration. We state it
under the assumption that the kernel $K:D\times D\to \C$ can be written as a sum of squares of reproducing kernels with constant diagonal.
To be more specific, we assume that $K$ can be written as
\begin{equation}\label{eq:Kdecomp}
K(x,y)=\sum_{i=1}^m |M_i(x,y)|^2\quad \text{for all}\ x,y\in D,
\end{equation}
where $M_i:D\times D\to \C$ are positive semi-definite functions on $D$, which are constant on the diagonal, i.e., $M_{\new{i}}(x,x)=c_i\ge 0$ for all $x\in D$.
Then $K(x,x)=\kappa:=\sum_{i=1}^m c_i^2$ for every $x\in D$ and also $K$ is constant on the diagonal.

\begin{thm}\label{thm:sum-of-squares}
Let $H$ be a RKHS on a domain $D$ with the kernel $K:D\times D\to \C$, 
which can be written as a sum of squares of reproducing kernels with constant diagonal,
and let $1\in H$. Consider the integration problem $S=S_h$ with the constant representer $h=1$.
Then
\[
 e_n(H,S)^2 \,\ge\, \Vert h \Vert_H^2 
  - \frac{n}{\kappa},
\]
where $\kappa$ is the value of $K$ on the diagonal. 
\end{thm}

\smallskip

\begin{proof}
Let $K$ be written as in \eqref{eq:Kdecomp}
and let $M_i(x,x)=c_i$.
Further, let $x_1,\hdots,x_n \in D$.
Then $(M_i(x_j,x_k))_{j,k=1}^n$ is positive semi-definite and by Proposition~\ref{prop:Schur}, 
so is the matrix
\[
\left(|M_i(x_j,x_k)|^2-\frac{c_i^2}{n} \right)_{j,k=1}^n. 
\]
Therefore, also the sum of these matrices over all $i\le m$, i.e., the matrix
\[
\left(K(x_j,x_k)-\frac{\kappa}{n} \right)_{j,k=1}^n 
\]
is positive semi-definite.
By Proposition~\ref{prop:charact}, together with $h=1$, this implies the error bound.
\end{proof}
\smallskip

In particular, one needs at least $\frac12 \Vert 1 \Vert_H^2 \cdot K(x,x)$ quadrature points
in order to reduce the initial error by a factor of two.
This insight is often already enough to prove the curse of dimensionality,
see \cite{HKNV21} and Section~\ref{sec:large}.
Surprisingly, Theorem~\ref{thm:sum-of-squares} can also be used
to prove new results on the order of convergence of the integration error.
This path is described in the next section.

\medskip

\begin{openprob} The Conjecture 2 of \cite{HinVyb11} suggests that, if $f:\R^d\to \R$
is non-negative and has a non-negative Fourier transform, then the matrix
\[
\left\{f(x_j-x_k)-\frac{f(0)}{n}\right\}_{j,k=1}^n
\]
is positive semi-definite for every $n\in\N$ and every choice of $\{x_1,\dots,x_n\}\subset \R^d.$
Let us note that Proposition \ref{prop:Schur} together with the classical Bochner theorem (cf. \cite[Theorem 5]{V20})
gives an affirmative answer to this conjecture if $f=g^2$, where $g:\R^d\to\R$ has a non-negative Fourier transform.
But in its full generality, the Conjecture 2 of \cite{HinVyb11} seems to be still open.
\end{openprob}

\section{Some general lower bounds for periodic functions}
\label{sec:general-periodic}

We now transfer Theorem \ref{thm:sum-of-squares} to the setting of periodic function spaces $H_\lambda$ on $D=[0,1]^d$.
We start with a result for sequences $\lambda \in \ell_1(\Z^d)$ 
which are given as a sum of convolution squares in Section~\ref{sec:sumsofsquares}.
We extend this result to the more general class of sequences
which can be written as a non-increasing function of a norm in Section~\ref{sec:normdec}.
As this covers all non-increasing sequences in the univariate case,
we obtain as a corollary a new and sharp result on the largest possible 
error $e_n(H)$
for any fixed sequence of approximation numbers 
in Section~\ref{sec:L2app}.

\subsection{Sums of squares}
\label{sec:sumsofsquares}

Recall, that
the reproducing kernel $K_\lambda$ of $H_\lambda$ for a non-negative and summable sequence $\lambda=(\lambda_k)_{k\in\Z^d}$ is given by \eqref{eq:kernel}.
The square of its absolute value is then given by
\begin{align}\label{eq:K_square}
|K_\lambda(x,y)|^2=\sum_{j,\ell\in\Z^d}\lambda_j \lambda_\ell e_{j-\ell}(x-y)=\sum_{\ell\in\Z^d}\lambda_\ell\sum_{k\in\Z^d}\lambda_{k+\ell}e_k(x-y).
\end{align}

Therefore, we define the convolution of two non-negative sequences $\lambda, \theta \in \ell_1(\Z^d)$ by
\[
 (\lambda \ast \theta)_k \,=\, \sum_{\ell \in \Z^d} \lambda_\ell\, \theta_{\ell+k},
 \quad k\in\Z^d.
\]
A straightforward calculation shows that $\lambda \ast \theta \in \ell_1(\Z^d)$ and that $\Vert \lambda \ast \theta \Vert_1 = \Vert \lambda \Vert_1 \cdot \Vert \theta \Vert_1$.
This notation allows us to reformulate \eqref{eq:K_square} as $|K_\lambda|^2=K_{\lambda\ast\lambda}$.
We say that $\lambda$ is a sum of convolution squares if there are $\lambda^{(i)} \in \ell_1(\Z^d)$, $i\le m$, such that $\lambda = \sum_{i\le m} \lambda^{(i)} \ast \lambda^{(i)}$.
Theorem~\ref{thm:sum-of-squares} then takes the following form.

\begin{cor}\label{cor:sum-of-squares}
If $\lambda \in \ell_1(\Z^d)$ is a sum of convolution squares,
then
\[
 e_n(H_\lambda)^2 \,\ge\, \lambda_0 \left( 1 - \frac{n\lambda_0}{\Vert \lambda\Vert_1} \right).
\]
\end{cor}

\begin{proof}
Since both sides of the inequality are homogeneous in $\lambda$, 
we may assume that $\lambda_0 =1$.
In this case, the representer of the integral on $H_\lambda$ is given by $h=1$,
where $\Vert h \Vert_{H_\lambda}^2 = 1$.
By \eqref{eq:K_square}, we obtain
\[
 K_\lambda \,=\, \sum_{i=1}^m K_{\lambda^{(i)} \ast \lambda^{(i)}}
 \,=\, \sum_{i=1}^m \big|K_{\lambda^{(i)}}\big|^2.
\]
Therefore, we may apply Theorem~\ref{thm:sum-of-squares} and simply have to note that $K_\lambda(x,x)=\Vert \lambda\Vert_1$.
\end{proof}

This is a generalization of \cite[Theorem 1]{HKNV22}
which covers the case $d=1$ and $m=1$.

\subsection{Norm-decreasing sequences}
\label{sec:normdec}

Our next step is to save Corollary~\ref{cor:sum-of-squares}
for more general sequences $\lambda$,
which are not given as a sum of convolution squares.
Namely, we consider sequences of the form $\lambda_k = g( \Vert k \Vert)$,
where $\Vert \cdot\Vert$ is a norm on $\R^d$
and $g\colon [0,\infty) \to [0,\infty)$ is monotonically decreasing.
We call such sequences \emph{$\Vert\cdot\Vert$-decreasing}.
Clearly, $\lambda$ is $\Vert\cdot\Vert$-decreasing if and only if 
it satisfies $\lambda_k \le \lambda_\ell$ for all $k,\ell\in\Z^d$ with $\Vert k \Vert \ge \Vert\ell\Vert$.
\medskip

\begin{thm}\label{thm:lambda-rad} 
Let $\lambda\in\ell_1(\Z^d)$ be $\Vert\cdot\Vert$-decreasing. Then
\begin{equation*}
e_n(H_\lambda)^2 \,\ge\, \lambda_0\left(1-\frac{2n\lambda_0}{\lambda_0 + \sum_{k\in  \Z^d} \lambda_{2k}}\right).
\end{equation*}
\end{thm}

\medskip

\begin{proof}
Again, since both sides of the stated inequality are homogeneous with respect to $\lambda$,
we may assume 
that $\sum_{k\in  \Z^d} \lambda_{2k}=1$.
We set $\mu_\ell = 2^{-1/2} \lambda_{2\ell}$.
By the triangle inequality, one of the two relations $2\Vert k\Vert \ge \Vert \ell\Vert$ or $2\Vert k+\ell\Vert \ge \Vert \ell\Vert$ must hold for each pair $k,\ell\in\Z^d$,
and therefore $\lambda_{2k} \le \lambda_\ell$ or else $\lambda_{2k+2\ell} \le \lambda_\ell$.
Thus we have for all $\ell\in\Z^d$ that 
\[
 (\mu \ast \mu)_\ell = \frac12 \sum_{k\in\Z^d} \lambda_{2k} \lambda_{2k+2\ell} \le \lambda_\ell.
\]
Moreover, $(\mu \ast \mu)_0 \le \lambda_0/2$.
We put $\nu = \mu \ast \mu + t \delta_0$ and choose $t\ge \lambda_0/2$ such that $\nu_0 = \lambda_0$.
Then $\nu$ is a sum of convolution squares. 
It follows from Corollary~\ref{cor:sum-of-squares} and $\nu \le \lambda$ that
\[
 e_n(H_\lambda)^2 \,\ge\,
 e_n(H_\nu)^2 \,\ge\,
 \lambda_0 \left( 1 - \frac{n\lambda_0}{\Vert \nu \Vert_1}\right)
\]
with
\[
 \Vert \nu \Vert_1 \ge \frac{\lambda_0}{2} + \Vert \mu \ast \mu \Vert_1 = \frac{\lambda_0}{2} + \Vert \mu \Vert_1^2 = \frac{\lambda_0+1}{2}.
\]
\end{proof}

Theorem~\ref{thm:lambda-rad} is well-suited to prove the curse of dimensionality,
see Section~\ref{sec:large}.
However, we can also tune it to be used 
for results on the asymptotic behavior
of the $n$th minimal error.
\medskip

\begin{thm}\label{thm:asymptotic-mulit}
Let $\lambda\in\ell_1(\Z^d)$ be $\Vert\cdot\Vert$-decreasing.
For $n\in\N$, 
we let $r_n$ be the
norm of the $(4n-1)$th element in a $\Vert\cdot\Vert$-increasing
rearrangement of $(2\Z)^d$, i.e.,
$$r_n := \min \left\{ r\ge 0 \mid \#\{ k\in \Z^d \colon \Vert 2k \Vert \le r\} \ge 4n-1\right\}.$$
Then 
\[
 e_{n}(H_\lambda)^2 \,\ge\, \min\bigg\{ \frac{\lambda_0}{2},\, \frac{1}{8n} \sum_{\Vert 2k\Vert \,>\, r_n} \lambda_{2k} \bigg\}.
\]
\end{thm}
\smallskip

Note that $r_n \asymp n^{1/d}$ for all norms $\Vert\cdot\Vert=\Vert\cdot\Vert_p$ with $1\le p \le \infty$.
\smallskip

\begin{proof}
Choose $m\in\Z^d$ with $\Vert m \Vert = r_n$.
We define $\tau \in \ell_1(\Z^d)$ by 
setting $\tau_k=\lambda_k$ for $\Vert k \Vert >r_n$
and $\tau_k=\lambda_m$ for $0<\Vert k \Vert \le r_n$
as well as
\[
\tau_0 \,=\, \min\bigg\{ \lambda_0,\, \max\bigg\{ \lambda_m,\, \frac{1}{4n} \sum_{\Vert 2k\Vert > r_n} \lambda_{2k}\bigg\} \bigg\}.
\]
Then $\tau$ is $\Vert\cdot\Vert$-decreasing and bounded above by $\lambda$.
Moreover,
\[
 \tau_0 + \sum_{k\in\Z^d} \tau_{2k} \,\ge\, 
 4n \lambda_m + \sum_{\Vert 2k\Vert > r_n} \lambda_{2k} \,\ge\, 4n\tau_0
\]
and thus Theorem~\ref{thm:lambda-rad}  gives
$e_n(H_{\lambda})^2 \ge e_n(H_{\tau})^2\ge \tau_0/2$,
which leads to the stated lower bound.
\end{proof}

In the univariate case, the previous result looks as follows.

\begin{cor}\label{cor:univariate}
Let $\lambda\in\ell_1(\Z)$ be non-negative, symmetric and monotonically decreasing on $\N_0$. 
Then
\[
e_n(H_\lambda)^2 \,\ge\, \min\left\{\frac{\lambda_0}{2},\, \frac{1}{8n}\sum_{k \,\ge\, 4n}\lambda_k\right\}
\qquad \text{for all }\, n\in\N.
\]
\end{cor}
\smallskip

\begin{proof}
We apply Theorem~\ref{thm:asymptotic-mulit} for $d=1$ and $\Vert\cdot\Vert=\vert\cdot\vert$.
Therefore, $r_n=4n-2$ 
and because of monotonicity and symmetry, 
we get
\[
\sum_{\|2k\|>r_n}\lambda_{2k}=\sum_{|k|\ge 2n}\lambda_{2k}=\sum_{k\ge 2n} (\lambda_{2k}+\lambda_{-2k})\ge \sum_{k\ge 2n} (\lambda_{2k}+\lambda_{2k+1})=\sum_{k\ge 4n} \lambda_k.
\]
\end{proof}

This improves upon our previous result \cite[Theorem~4]{HKNV22},
where we obtained a similar lower bound
but only under an additional regularity assumption on the sequence $\lambda$.
\bigskip

\subsection{Detour: The power of function values for $L_2$ approximation}
\label{sec:L2app}

Recently, there has been an increased interest in the comparison 
of standard information given by function values and general linear information
for the problem of $L_2$ approximation.
We refer to \cite{DKU,KU19,KU21,NSU,T20} for recent upper bounds
and to \cite{HKNV21,HKNV22} for lower bounds. 
\new{Let us denote by $\Omega$ the set of all pairs $(H,\mu)$ 
consisting of a separable RKHS $H$ on an arbitrary set $D$ and a measure $\mu$ on $D$ such that 
$H$ is embedded into $L_2(D,\mu)$.}
For $(H,\mu)\in\Omega$, we define the sampling numbers
\[
 g_n(H,\mu) \,=\, \inf_{\substack{x_1,\dots,x_n\in D\\g_1,\dots,g_n \in L_2}}\,
 \sup_{\|f\|_{H}\le 1}\bigg\|f-\sum_{i=1}^n f(x_i)g_i\,\bigg\|_{L_2(D,\mu)}
\]
and the approximation numbers
\[
 a_n(H,\mu) \,=\, \inf_{\substack{L_1,\dots,L_n\in H'\\g_1,\dots,g_n \in L_2}}\,
 \sup_{\|f\|_{H}\le 1}\bigg\|f-\sum_{i=1}^n L_i(f)g_i\,\bigg\|_{L_2(D,\mu)}.
\]
One is interested in the largest possible gap between the two concepts,
that is, given a sequence $\sigma_0\ge \sigma_1\ge \hdots$ and an integer $n \ge 0$, one considers
\[
 g_n^*(\sigma) \,:=\, \sup \Big\{ g_n(H,\mu) \,\Big|\, 
(H,\mu)\in\Omega \colon \forall m\in\N_0 \colon a_m(H,\mu) = \sigma_m \Big\}.
\]
It is known from \cite{HNV} that $g_n^*(\sigma)=\sigma_0$,
whenever $\sigma\not\in\ell_2$.
On the other hand, it was proven in \cite{DKU} that
there is a universal constant $c\in\N$ such that, whenever $\sigma\in\ell_2$,
\[
 g_{cn}^*(\sigma) \,\le\, \sigma_n^* \,:=\, \min\left\{\sigma_0,\, \sqrt{\frac1n \sum_{k\ge n} \sigma_k^2} \right\}.
\]
We obtain a matching lower bound as a consequence of Corollary~\ref{cor:univariate}.
For the spaces $H_\lambda$,
the sequence of the squared approximation numbers equals
the non-increasing rearrangement of the sequence $\lambda$.
Here, we use that approximation on $H_\lambda$ is harder than integration, 
namely, 
$e_n(H_\lambda)\le g_n(H_\lambda)$.
Indeed, if $S_n=\sum_{i=1}^n g_i \delta_{x_i}$ 
is a sampling operator, then we consider the quadrature formula
$Q_n(f)=\int_0^1 S_n(f)(x) \,{\rm d}x$ and obtain
\[
e(Q_n)\,=\,\sup_{\|f\|_{H_\lambda}\le 1}\left|\int_0^1 (f-S_n(f))(x)\,{\rm d}x\right|
\,\le\, \sup_{\|f\|_{H_\lambda}\le 1} \|f-S_n(f)\|_2
\,=\, e(S_n).
\]
We apply Corollary~\ref{cor:univariate} for the Hilbert spaces $H_\lambda$ with $\lambda_k=\sigma_{|2k|}^2$ 
and obtain
\[
 g_n^*(\sigma) \,\ge\, \frac14 \sigma_{8n}^*. 
\]
\new{This improves upon the currently best known lower bound from \cite[Theorem~2]{HKNV22}
in the sense that our lower bound holds for all and not just infinitely many $n\in\N_0$.
Moreover, due to} $e_n(H) \le g_n(H)$ and the fact that our lower bounds are proven for the integration problem, 
an analogous result holds with $g_n^*(\sigma)$
replaced by 
\[
 e_n^*(\sigma) \,:=\, \sup \Big\{ e_n(H,\INT_\mu) \,\Big|\, 
(H,\mu) \in \Omega_0 \colon \forall m\in\N_0 \colon a_m(H,\mu) = \sigma_m \Big\},
\]
where $\Omega_0$ is the set of pairs $(H,\mu) \in \Omega$ such that $\mu$ is a probability measure.
Thus, we have the following corollary.
\medskip

\begin{cor}\label{cor:opt}
There are universal constants $0<c<1<C$
such that, for any sequence $\sigma_0 \ge \sigma_1 \ge \hdots$
and any integer $n\ge 0$, we have
\[
 c\, \sigma_{Cn}^* \,\le\, e_n^*(\sigma) \,\le\, g_n^*(\sigma) \,\le\, C \sigma^*_{\lfloor cn\rfloor}.
\]
\end{cor}
\medskip

In this sense,
the worst possible behavior of the sampling numbers (or the minimal integration error)
for a given sequence of approximation numbers $\sigma$
is always described by the sequence $\sigma^*$.
If $\sigma$ is regularly decreasing in the sense that $\sigma_n \asymp \sigma_{2n}$,
we obtain that 
\[
g_n^*(\sigma) \,\asymp\, e_n^*(\sigma) \,\asymp\, \sigma_n^*.
\]
\smallskip

Let us consider the case of polynomial decay, that is, $\sigma_n \asymp n^{-\alpha} \log^{-\beta} n$.
This sequence is square-summable if and only if $\alpha>1/2$
or $\alpha=1/2$ and $\beta>1/2$.
In the case $\alpha>1/2$ it follows that $g_n^*(\sigma) \asymp \sigma_n^* \asymp \sigma_n$.
For the reproducing kernel Hilbert spaces $H_\lambda$ of multivariate periodic functions 
the sequence of approximation numbers is always square-summable.
Thus there can only be a gap between the concepts of sampling and approximation numbers 
in the case $\alpha=1/2$.
This corresponds to function spaces 
\new{of small smoothness}
which are discussed in the next section.

\section{Lower bounds for small smoothness}
\label{sec:small}

In this section, we consider various function spaces of \new{small}
smoothness.
Spaces of that type appeared already in \cite{Levy} to characterize path regularity of the Wiener process. We refer also to \cite{KSchV}
to recent results on this subject and to \cite{DT} for an extensive treatment of function spaces with logarithmic smoothness.
These spaces are of particular interest to us, since 
\begin{itemize}
\item this is the only case where an asymptotic gap between the approximation numbers and the sampling numbers is possible, see Section~\ref{sec:L2app}.
\item the standard technique of bump functions does not yield optimal lower bounds, see Section~\ref{sec:limits}.
\end{itemize}
In \cite{HKNV22},
we obtained lower bounds for the 
univariate Sobolev spaces \new{which merge fractional smoothness $1/2$ and logarithmic smoothness}.
We want to extend these results to the multivariate case.
In the multivariate regime, there are different smoothness
scales that generalize the univariate smoothness scale.
We consider spaces of isotropic smoothness
and spaces of mixed smoothness.

\subsection{Isotropic smoothness}

In the case of isotropic smoothness,
we consider sequences of the form
\begin{equation}\label{eq:lambda-iso}
 \lambda_k = (1+|k|)^{-d} \log^{-2\beta}(e+|k|), \quad k \in \Z^d, \quad \beta>1/2,
\end{equation}
where $\vert \cdot\vert$ denotes the Euclidean norm on $\Z^d$.
This sequence is $\vert \cdot\vert$-decreasing and we may therefore apply Theorem~\ref{thm:asymptotic-mulit} to obtain the following result.
Recall that the approximation numbers and sampling numbers are defined in Section~\ref{sec:L2app}.

\begin{cor}\label{cor:lowerbound}
Let $\lambda$ be given by \eqref{eq:lambda-iso}.
Then
\[
a_n(H_\lambda) \,\asymp\, n^{-1/2} \log^{-\beta} n
\]
and
\[
e_n(H_\lambda)
\,\asymp\,
g_n(H_\lambda)
\,\asymp\, n^{-1/2} \log^{-\beta+1/2} n.
\]
\end{cor}

\begin{proof}
Recall that $a_n(H_\lambda)^2$ is the $(n+1)$st largest entry of $\lambda$.
We have
\[
 \#\{k\in\Z^d \colon \lambda_k \ge \varepsilon\} \,\asymp\, \varepsilon^{-1} \big(\log \varepsilon^{-1}\big)^{-2\beta}
\]
for $\varepsilon\to 0^+$
and this easily implies the asymptotic behavior of $a_n(H_\lambda)$.
Recalling that $e_n(H_\lambda) \le g_n(H_\lambda)$,
the upper bounds on $e_n(H_\lambda)$ and  $g_n(H_\lambda)$ 
follow from \cite[Theorem~1]{DKU}
and
\begin{equation}\label{eq:sumbehavior}
 \sum_{k\ge n} k^{-1} \log^{-2\beta} k \,\asymp\, \log^{-2\beta+1} n.
\end{equation}
The lower bounds follow from Theorem~\ref{thm:asymptotic-mulit}
and \eqref{eq:sumbehavior}
as the condition $|2k|>r_n$ excludes only 
$\mathcal O(n)$ approximation numbers from the sum
in the lower bound of Theorem~\ref{thm:asymptotic-mulit}.
\end{proof}

\new{We remark that, without much additional work,
the norm characterization from Proposition \ref{prop:diff} can be generalized to $H_\lambda$ with $\lambda$ given by \eqref{eq:lambda-iso}.
In this case we need differences of higher order, which are defined for $h,x\in\R^d$
inductively by
\[
\Delta^1_hf(x)=f(x+h)-f(x),\quad \Delta^{j+1}_h f(x)=\Delta^1_h(\Delta^j_h f)(x),\quad j\ge 1.
\]
Using this notation, the multivariate counterpart of Proposition \ref{prop:diff} then reads as follows.

\begin{prop}\label{prop:eqnormd}
Let $\lambda\in\ell_1(\Z^d)$ be given by \eqref{eq:lambda-iso} and let $M>d/2$ be an integer.
Then $f\in C([0,1]^d)$ belongs to $H_\lambda$ if, and only if,
\begin{equation}\label{eq:Delta'}
\|f\|_2^2+\int_{h\colon |h|\le 1} \frac{[1-\log |h|]^{2\beta}}{|h|^{2d}}\cdot \|\Delta^M_h f\|_2^2\ dh
\end{equation}
is finite. Furthermore, $\|f\|_{H_\lambda}^2$ is equivalent to \eqref{eq:Delta'} with the constants independent of $f\in C([0,1]^d)$.
\end{prop}
The proof resembles very much the proof of Proposition~\ref{prop:diff} and we leave out the rather technical details, cf.\ also \cite[Theorem 10.9]{DT},
\new{where one can find a more general characterization for function spaces defined on whole $\R^d$.
Note that in contrast to the univariate case \eqref{eq:Delta}, where $h$ was from the unit interval $(0,1)$, we now consider in \eqref{eq:Delta'} all directions $h$ from the unit ball of $\R^d.$}
Similar to the univariate case, using Proposition~\ref{prop:eqnormd} instead of Proposition~\ref{prop:diff},
one can show that the bump function technique would not suffice to prove the lower bound on $e_n(H_\lambda)$ in Corollary~\ref{cor:lowerbound}.}

\begin{openprob} A logarithmic gap between upper and lower bounds of the worst-case error for numerical integration was recently observed also in \cite{GS19}
for Sobolev spaces of functions on the unit sphere ${\mathbb S}^d\subset {\mathbb R}^{d+1}.$ As conjectured already in \cite{GS19}, we also believe that the existing lower
bound can be improved. Unfortunately, our results can not be directly applied in this setting, because the norm of a function in these function spaces is defined
in terms of its decomposition into the orthonormal basis of spherical harmonics instead of the trigonometric system. Still, it might be possible to transfer our results to the sphere by
\begin{itemize}
\item showing that our lower bounds from Corollary~\ref{cor:lowerbound}
already hold for the subspace $H_\lambda^\circ$ of functions with compact support in $(0,1)^d$ and
\item establishing an equivalent characterization of the spaces of \new{generalized} smoothness on the sphere
using a decomposition of unity and lifting of the spaces $H_\lambda^\circ$ in analogy to \cite[Section~27]{Tri92},
\end{itemize}
or alternatively, by working with Theorem~\ref{thm:sum-of-squares} directly and a closer examination of 
(sums of) squares of kernels of Sobolev spaces on the sphere.
For the first approach, Proposition~\ref{prop:eqnormd} might help.
For the second approach, the paper \cite{Gasper} might be useful.
\end{openprob}

\subsection{Mixed smoothness}
\label{sec:mix}

We consider \new{spaces of small mixed smoothness}.
The space is defined as the $d$-fold tensor product of the univariate space of mixed smoothness from Section~\ref{sec:limits}.
This results in the space $H_\lambda$ with 
\begin{equation}\label{eq:lambda-mix}
 \lambda_k \,=\, \prod_{j=1}^d (1+|k_j|)^{-1} \log^{-2\beta}(e+|k_j|), \quad k\in\Z^d, \quad \beta>1/2.
\end{equation}

Here, $\lambda$ is not norm-decreasing
and therefore we cannot use Theorem~\ref{thm:asymptotic-mulit}.
However, it will turn out that already the lower bound from the univariate space
is sharp in this case.
The approximation numbers for $d>1$ have the same asymptotic behavior as in the case $d=1$.

\begin{thm}\label{thm:appnbs}
Let $\lambda$ be given by \eqref{eq:lambda-mix}. Then
\[
a_n(H_\lambda) \,\asymp\, n^{-1/2} \log^{-\beta} n.
\]
\end{thm}

\medskip

This is in sharp contrast to the spaces of mixed smoothness $s>1/2$,
where the approximation numbers for $d>1$ have a lower speed of convergence than for $d=1$,
see for example \cite[Theorem 4.45]{DTU}.
The proof of Theorem~\ref{thm:appnbs} is based on the following combinatorial lemma,
which is in contrast to \cite[Lemma~3.2]{KSU15}.
For $r\ge 1$ and $d\in\N$, we denote
\[
 M(r,d) \, := \, \bigg\{ (n_1,\hdots,n_d)\in\Z^d 
 \,\bigg|\, \prod_{j=1}^d (1+|n_j|) \log^{2\beta}(e+|n_j|) \le r \bigg\}
\]
and $N(r,d) = \# M(r,d)$.
\medskip

\begin{lem}\label{lem:count}
For fixed $d\in\N$,
\[
 N(r,d) \,\lesssim\, r \log^{-2\beta}(e+r).
\]
\end{lem}
\smallskip

\begin{proof}
We prove the statement by induction.
Clearly, the statement is true for $d=1$.
Let $d>1$ and let the statement be true for $N(r,d-1)$.
Then
\begin{multline*}
 N(r,d) = \sum_{n \in M(r,1)} N\bigg(\frac{r}{(1+|n|) \log^{2\beta}(e+|n|)} , d-1 \bigg) \\
 \lesssim \sum_{n \in M(r,1)} \frac{r}{(1+|n|) \log^{2\beta}(e+|n|)} \log^{-2\beta}\bigg(e + \frac{r}{(1+|n|) \log^{2\beta}(e+|n|)} \bigg).
\end{multline*}
In the case $|n|\le \sqrt{r}$, we have
\[
\log^{-2\beta}\bigg( e + \frac{r}{(1+|n|) \log^{2\beta}(e+|n|)} \bigg)
\,\lesssim\, \log^{-2\beta}(e+r)
\]
and thus
\begin{multline*}
 \sum_{|n| \le \sqrt r } \frac{r}{(1+|n|) \log^{2\beta}(e+|n|)} \log^{-2\beta}\bigg( e + \frac{r}{(1+|n|) \log^{2\beta}(e+|n|)} \bigg) \\
 \lesssim\,  r \log^{-2\beta}(e+r)  
 \sum_{n\in\Z} \frac{1}{(1+|n|) \log^{2\beta}(e+|n|)}
 \,\lesssim\,   r \log^{-2\beta}(e+r).
\end{multline*}
In the case $|n|\ge \sqrt{r}$, $n\in M(r,1)$, we have 
$\log^{2\beta}(e+|n|)\asymp \log^{2\beta}(e+r)$
and thus,
\begin{multline*}
 \sum_{|n| \ge \sqrt r } \frac{r}{(1+|n|) \log^{2\beta}(e+|n|)} 
 \log^{-2\beta}\bigg( e + \frac{r}{(1+|n|) \log^{2\beta}(e+|n|)} \bigg) \\
 \lesssim  \int_{\sqrt r }^{C r \log^{-2\beta}( e+r)} 
 \underbrace{\frac{r}{n \log^{2\beta}(e+r)}}_{=: u} 
 \log^{-2\beta}\bigg(e+ \frac{r}{n \log^{2\beta}(e+r)} \bigg) \,dn\\
 \le\, \int_{1/C}^{\infty} u \log^{-2\beta}( e+u ) \,\frac{r \log^{-2\beta}( e+r)}{u^2} \,du 
  \,\lesssim\, r \log^{-2\beta}( e+r).
\end{multline*}
\end{proof}

\begin{proof}[Proof of Theorem~\ref{thm:appnbs}]
The sequence of approximation numbers $a_{n,d}:=a_n(H_\lambda,L_2)$ is the decreasing rearrangement of the sequence $(\sqrt{\lambda_k})_{k\in\Z^d}$.
With $\lambda_0 =1$, we get $a_{n,d} \ge \sqrt{\lambda_{(n,0,...,0)}}$ and the lower bound is obvious.
The upper bound is obtained from Lemma~\ref{lem:count}.
Given $n\ge 3$, we choose $r=r(n) \asymp n\log^{2\beta}n$ such that $N(r,d) \le n$.
Since
\[
 N(r,d) = \#\{ k\in\Z^d \mid \lambda_k \ge r^{-1}\}
 = \#\{m\in\N_0 \mid a_{m,d} \ge r^{-1/2} \}
\]
we have $a_{n,d} < r(n)^{-1/2}$ 
and the statement is proven.
\end{proof}

\begin{openprob}
It would be interesting to determine the so-called \emph{asymptotic constants}
\[
\limsup_{n\to\infty} \frac{a_n(H_\lambda)}{n^{1/2} \log^{\beta} n}
\qquad\text{and}\qquad
\liminf_{n\to\infty} \frac{a_n(H_\lambda)}{n^{1/2} \log^{\beta} n}
\]
in the case of
\new{mixed smoothness $1/2$ with logarithmic perturbation}.
Opposed to Theorem~\ref{thm:appnbs}, the asymptotic constants might reveal a dependence
of the asymptotic decay upon the dimension $d$.
The asymptotic constants for mixed smoothness $s>1/2$ have been determined in \cite{KSU15}.
In contrast to larger smoothness,
the asymptotic constants for
\new{mixed smoothness $1/2$ with logarithmic perturbation}
cannot decay as a function of $d$.
\end{openprob}

Now, we immediately obtain the behavior of the sampling numbers and the integration error
on the spaces of
\new{mixed smoothness $1/2$ with logarithmic perturbation}.

\begin{thm}\label{thm:samplingmix}
Let $\lambda$ be given by \eqref{eq:lambda-mix}. Then
\[
g_n(H_\lambda) 
\,\asymp\, e_n(H_\lambda)
\,\asymp\, n^{-1/2} \log^{-\beta+1/2} n.
\]
\end{thm}

\begin{proof}
The upper bound follows from Theorem~\ref{thm:appnbs} and \cite[Corollary~2]{DKU}.
The proof of the lower bound is essentially transferred from the one-dimensional case.
For that sake, we denote by $\overline\lambda=(\overline\lambda_k)_{k\in\Z}$ the sequence introduced in \eqref{eq:border-lambda}.
Note that $\overline \lambda_k=\lambda_{ke_1}$, where $e_1=(1,0,\dots,0)$ is the first canonical unit vector.
To prove now the lower bound for $\INT$, let $Q_n$ be a quadrature formula on $H_{\lambda}$
with nodes $x_1,\dots, x_n\in [0,1]^d$.
For $x\in [0,1]^d$ and $k\le d$, let $x^{(k)} \in [0,1]$ denote the $k$th coordinate of $x$.
With $Q_n^{(1)}$ we denote the quadrature rule with nodes $x_1^{(1)},\dots, x_n^{(1)}\in [0,1]$ and the same weights as in $Q_n$.
By Corollary \ref{cor:lowerbound} or \cite[Theorem~4]{HKNV22} 
there is a function $f^{(1)}\colon [0,1] \to \C$ in the unit ball of $H_{\overline{\lambda}}$ such that 
$$\bigg|Q_n^{(1)}(f^{(1)}) - \int_0^1 f^{(1)}(x)\,dx\bigg| \,\gtrsim\, n^{-1/2} \log^{-\beta+1/2} n.$$
The function
\[
 f \colon [0,1]^d \to \C, \quad f(x) = f^{(1)}(x^{(1)})
\]
is contained in the unit ball of $H_{\lambda}$
and satisfies $Q_n(f)=Q_n^{(1)}(f^{(1)})$ and
$\int_{[0,1]^d} f(x)\,dx=\int_0^1 f^{(1)}(x)\,dx$.
Thus,
$$\bigg|Q_n(f) - \int_{[0,1]^d} f(x)\,dx\bigg| \,\gtrsim\, n^{-1/2} \log^{-\beta+1/2} n.$$
\end{proof}

\begin{openprob}
As \cite[Corollary~2]{DKU} is only an existence result,
the upper bound on the integration error in Theorem~\ref{thm:samplingmix} is not constructive
and does not tell us how to choose the quadrature points optimally.
The results from \cite{KU19,Ull} show that $\mathcal O(n\log n)$ i.i.d.\ uniformly distributed
quadrature points are suitable with high probability to achieve an error of order $e_n(H_\lambda)$.
In the case of isotropic smoothness $s>d/2$,
it is known that $\mathcal O(n)$ such points suffice, see \cite{KS},
and the question arises whether the same holds true for
other spaces from the family $H_\lambda$, $\lambda \in \ell_1(\Z^d)$.
\end{openprob}

\section{Lower bounds for large smoothness}
\label{sec:large}

We discussed in Section \ref{sec:small} that the bump-function technique does not provide
optimal lower bounds for numerical integration of functions from $H_{\lambda}$ as introduced in Definition \ref{dfn:Hper}
if the sequence $\lambda=(\lambda_k)_{k\in\Z^d}$ is large, i.e., if it is barely square-summable.
Quite naturally, the technique fails also in the other extremal regime - namely when the sequence $\lambda$ decays very rapidly.
Then the space $H_\lambda$ consists only of analytic functions and, therefore, contains no bump functions with compact support at all.

Function spaces of analytic functions have a long history. 
Nevertheless, it is an active research question
whether or not their use can help to
avoid the curse of dimension for numerical integration and approximation. 
Based on the technique of \cite{SW97}, the more recent papers \cite{KSW17} and \cite{IKLP15} present
lower bounds for numerical integration of such classes of functions which avoid any use of bump functions.
Here, we prove the curse of dimension for a class of analytic functions that is in connection with
a recent result of \cite{GODA}.
In the paper \cite{GODA},
it was shown that the integration problem on
\[
 F_d^1 \,=\, \left\{ f \in L_2([0,1]^d) \,\Big|\, \sum_{k\in\Z^d} \vert \hat f(k)\vert \cdot g(\Vert k\Vert_\infty) \le 1 \right\}
\]
is polynomially tractable already for very slowly increasing functions $g\colon \N_0\to (0,\infty)$,
namely, for $g(k)=\max(1,\log (k))$.
We contrast this very nice result by showing that the integration problem
suffers from the curse of dimension for essentially any non-trivial function $g$
if we replace $|\hat f(k)|$ by $|\hat f(k)|^2$ in the definition of $F_d^1.$

\begin{thm}\label{thm:curse} Let $d\ge 2$ and let $g:\N_0\to [0,\infty]$ be any non-decreasing function with $g(2)\le \tau<\infty$. If $4n-1\le 3^d$, then the worst-case
error for the numerical integration on the class
\[
 F_d^2 \,=\, \left\{ f \in L_2([0,1]^d) \,\Big|\, \sum_{k\in\Z^d} \vert \hat f(k)\vert^2 \cdot g(\Vert k\Vert_\infty) \le 1 \right\}
\]
satisfies $e_n(F_d^2)^2\ge 1/(2\tau)$. Hence, numerical integration suffers from the curse of dimension on the classes $F_d^2.$
\end{thm}
\begin{proof}
We identify the class $F_d^2$ with $H_\lambda$, where $\lambda_k=1/g(\|k\|_\infty)$. Then we simply apply Theorem \ref{thm:lambda-rad}
and obtain
\[
e_n(F_d^2)^2\ge \lambda_0\left(1-\frac{2n\lambda_0}{\lambda_0+\sum_{k\in\{-1,0,1\}^d}\lambda_{2k}}\right)
\ge \frac{1}{\tau}\left(1-\frac{2n}{1+3^d}\right)\ge\frac{1}{2\tau}.
\]
\end{proof}

We also add an asymptotic lower bound for analytic functions.
We only write down the univariate case for simplicity.
Recall that all the functions in $H_\lambda$ are analytic if
the sequence $\lambda$ decays geometrically
and that a lower bound on $e_n(H_\lambda)$ in this case 
cannot possibly be proven with the bump function technique.
We therefore write down the lower bound obtained with the Schur technique. 
We note, however, that a similar lower bound might be proven in this case
with the technique from \cite{SW97}.

\begin{cor}
Let $\lambda_k \ge c\, \omega^{-|k|}$ for some $c>0$, $\omega>1$ and all $k\in \Z$.
Then
\[
 e_{n}(H_\lambda)^2 \,\ge\, \frac{c}{2}\cdot\min\left\{1, \frac{\omega^{-4n}}{4n}\cdot\frac{1}{1-\omega^{-1}}\right\}.
\]
\end{cor}

\begin{proof}
This follows immediately from Corollary~\ref{cor:univariate}.
\end{proof}

\new{\bibitem{Kri23} D.~Krieg,
\emph{Tractability of sampling recovery on unweighted function classes}, arXiv:2304.14169, 2023.}

\bibitem{KU19} 
D.~Krieg and M.~Ullrich,
\emph{Function values are enough for $L_2$-approximation},
Found. Comput. Math. 21, 1141--1151, 2021. 

\bibitem{KU21} 
D.~Krieg and M.~Ullrich,
\emph{Function values are enough for $L_2$-approximation, part II,}
J. Complexity 66, 101569, 2021.

\bibitem{KS} 
D.~Krieg and M.~Sonnleitner,
\emph{Random points are optimal for the approximation of Sobolev functions},
\new{to appear in IMA J. Numer. Anal.}

\bibitem{KSU15} T. K\"uhn, W. Sickel, and T. Ullrich, 
\emph{Approximation of mixed order Sobolev functions on the d-torus: asymptotics, preasymptotics, and d-dependence},
Constr. Approx. 42(3), 353--398, 2015.

\bibitem{KSW17}
F. Y. Kuo, I. H. Sloan, and H. Wo\'zniakowski,
\emph{Multivariate integration for analytic functions with Gaussian kernels},
Math. Comp. 86, no. 304, 829–853, 2017.

\bibitem{Levy}
P.~L\'evy, Th\'eorie de l’Addition des Variables Al\'eatoires,
Monographies des Probabilit\'es; calcul des probabilit\'es et ses applications,
publi\'ees sous la direction de E. Borel, no. 1. Paris: Gauthier-Villars, 1937.

\bibitem{NSU} N. Nagel, M. Sch\"afer, and T. Ullrich,
{\it A new upper bound for sampling numbers}, 
Found. Comput. Math. 22(2), 445--468, 2022.

\bibitem{Nov88} E. Novak, Deterministic and Stochastic Error Bounds in Numerical Analysis, Lecture Notes in Mathematics 1349,
Springer-Verlag, Berlin, 1988.

\bibitem{NT06} E. Novak and H. Triebel, \emph{Function spaces in Lipschitz domains and optimal rates of convergence for sampling},
Constr. Approx. 23, 325--350, 2006.

\bibitem{NW01} E. Novak and H.~Wo\'zniakowski, 
\emph{Intractability results for integration and discrepancy},
J. Complexity 17(2), 388--441, 2001. 

\bibitem{NW1}
E.~Novak and H.~Wo\'zniakowski,
Tractability of Multivariate Problems,
Volume I: Linear Information, 
European Mathematical Society, Zürich, 2008.

\bibitem{NW2}
E.~Novak and H.~Wo\'zniakowski,
Tractability of Multivariate Problems,
Volume II: Standard Information for Functionals, 
European Mathematical Society, Zürich, 2010.

\bibitem{SW97} I. H. Sloan and H. Wo\'zniakowski, \emph{An intractability result for multiple integration},
Math. Comp., 66, 1119--1124, 1997.

\bibitem{Tem93}
V.~N. Temlyakov, Approximation of \new{P}eriodic \new{F}unctions,
Computational Mathematics and Analysis Series, 
Nova Science Publishers, Inc., Commack, NY, 1993.
	
\bibitem{T20} V. N. Temlyakov, {\it On optimal recovery in $L_2$},  
J. Complexity, 65, 101545, 2020. 

\bibitem{TWW88}  J. F. Traub, G. W. Wasilkowski, and H. Wo\'zniakowski, Information-Based Complexity, Academic Press, New York, 1988.

\bibitem{Tri92} H. Triebel, Theory of \new{F}unction \new{S}paces II, Birkhäuser Verlag,
Basel, 1992.

\bibitem{Ull} M. Ullrich, \emph{On the worst-case error of least squares algorithms for $L_2$-approximation with high probability},
J. Complexity 60, 101484, 2020.

\bibitem{Vyb08} J. Vyb\'\i ral, \emph{Dilation operators and sampling numbers}, J. Funct. Spaces Appl. 6, 17--46, 2008.

\bibitem{V20} J. Vyb\'\i ral, \emph{A variant of Schur's product theorem and its applications},
Adv. Math. 368, 107140, 2020.

\end{document}